\documentclass[11pt,reqno]{amsart}
\usepackage{amsmath,amssymb}

 \makeatletter
 \evensidemargin  \oddsidemargin
 \makeatother

\begin{document}
\thispagestyle{empty}
 \setcounter{page}{1}

\begin{center}
{\large\bf Stability of a  functional equation deriving from quartic
and additive functions

\vskip.20in

{\bf  M. Eshaghi Gordji } \\[2mm]

{\footnotesize Department of Mathematics,
Semnan University,\\ P. O. Box 35195-363, Semnan, Iran\\
[-1mm] e-mail: {\tt maj\_ess@Yahoo.com}} }
\end{center}
\vskip 5mm
 \noindent{\footnotesize{\bf Abstract.} In this paper, we obtain the general solution and the generalized
  Hyers-Ulam Rassias stability of the functional equation
$$f(2x+y)+f(2x-y)=4(f(x+y)+f(x-y))-\frac{3}{7}(f(2y)-2f(y))+2f(2x)-8f(x).$$
 \vskip.10in

  \newtheorem{df}{Definition}[section]
  \newtheorem{rk}[df]{Remark}
   \newtheorem{lem}[df]{Lemma}
   \newtheorem{thm}[df]{Theorem}
   \newtheorem{pro}[df]{Proposition}
   \newtheorem{cor}[df]{Corollary}
   \newtheorem{ex}[df]{Example}

 \setcounter{section}{0}
 \numberwithin{equation}{section}

\vskip .2in

\begin{center}
\section{Introduction}
\end{center}
The stability problem of functional equations originated from a
question of Ulam [24] in 1940, concerning the stability of group
homomorphisms. Let $(G_1,.)$ be a group and let $(G_2,*)$ be a
metric group with the metric $d(.,.).$ Given $\epsilon >0$, dose
there exist a $\delta
>0$, such that if a mapping $h:G_1\longrightarrow G_2$ satisfies the
inequality $d(h(x.y),h(x)*h(y)) <\delta,$ for all $x,y\in G_1$, then
there exists a homomorphism $H:G_1\longrightarrow G_2$ with
$d(h(x),H(x))<\epsilon,$ for all $x\in G_1?$ In the other words,
Under what condition dose there exists a homomorphism near an
approximate homomorphism? The concept of stability for functional
equation arises when we replace the functional equation by an
inequality which acts as a perturbation of the equation. In 1941, D.
H. Hyers [9] gave a first affirmative  answer to the question of
Ulam for Banach spaces. Let $f:{E}\longrightarrow{E'}$ be a mapping
between Banach spaces such that
$$\|f(x+y)-f(x)-f(y)\|\leq \delta, $$
for all $x,y\in E,$ and for some $\delta>0.$ Then there exists a
unique additive mapping $T:{E}\longrightarrow{E'}$ such that
$$\|f(x)-T(x)\|\leq \delta,$$
for all $x\in E.$ Moreover if $f(tx)$ is continuous in t for each
fixed $x\in E,$ then $T$ is linear. Finally in 1978, Th. M. Rassias
[21] proved the following Theorem.

\begin{thm}\label{t1} Let $f:{E}\longrightarrow{E'}$ be a mapping from
 a norm vector space ${E}$
into a Banach space ${E'}$ subject to the inequality
$$\|f(x+y)-f(x)-f(y)\|\leq \epsilon (\|x\|^p+\|y\|^p), \eqno \hspace {0.5
 cm} (1.1)$$
for all $x,y\in E,$ where $\epsilon$ and p are constants with
$\epsilon>0$ and $p<1.$ Then there exists a unique additive mapping
$T:{E}\longrightarrow{E'}$ such that
$$\|f(x)-T(x)\|\leq \frac{2\epsilon}{2-2^p}\|x\|^p,  \eqno \hspace {0.5
 cm}(1.2)$$ for all $x\in E.$
If $p<0$ then inequality (1.1) holds for all $x,y\neq 0$, and (1.2)
for $x\neq 0.$ Also, if the function $t\mapsto f(tx)$ from $\Bbb R$
into $E'$ is continuous for each fixed $x\in E,$ then T is linear.
\end{thm}
In 1991, Z. Gajda [5] answered the question for the case $p>1$,
which was rased by Rassias.  This new concept is known as
Hyers-Ulam-Rassias stability of functional equations (see [1,2],
[5-11], [18-20]).

In [15], Won-Gil Prak and Jea Hyeong Bae, considered the following
functional equation:
$$f(2x+y)+f(2x-y)=4(f(x+y)+f(x-y))+24f(x)-6f(y).\eqno \hspace {0.5cm}(1.3)$$
In fact they proved that a function
 f between real vector spaces X and Y is a solution of (1.3) if and only if there
 exists a unique symmetric multi-additive function $B:X\times X\times X\times X\longrightarrow Y$ such that
 $f(x)=B(x,x,x,x)$ for all $x$ (see [3,4], [12-17], [22,23]). It is easy to show that
 the function $f(x)=x^4$ satisfies the functional equation (1.3), which is called a
 quartic functional equation and every solution of the quartic functional equation is said to be a
 quartic function.

We deal with the next functional equation deriving from quartic and
 additive functions:
$$f(2x+y)+f(2x-y)=4(f(x+y)+f(x-y))-\frac{3}{7}(f(2y)-2f(y))+2f(2x)-8f(x).\eqno \hspace {0.5cm}(1.4)$$
It is easy to see that
 the function $f(x)=ax^4+bx$ is a solution of the functional equation (1.4). In the
 present paper we investigate the general solution and the generalized
  Hyers-Ulam-Rassias stability of the functional equation (1.4).

\vskip 5mm \begin{center}
\section{ General solution}
\end{center}
In this section we establish the general solution of functional
equation (1.4).
\begin{thm}\label{t1} Let $X$,$Y$
be vector spaces,  and let  $f:X\longrightarrow Y$  be a function
satisfies (1.4). Then the following assertions hold.

a) If f is even function, then f is quartic.

b) If f is odd function, then f is additive.
\end{thm}
\begin{proof} a) Putting $x=y=0$ in (1.4), we get $f(0)=0$. Setting $x=0$ in (1.4), by evenness of f, we obtain

$$f(2y)=16f(y), \eqno \hspace {0.5cm}(2.1)$$
for all $y\in X.$ Hence (1.4) can be written as

$$f(2x+y)+f(2x-y)=4(f(x+y)+f(x-y))+24f(x)-6f(y) \eqno \hspace {0.5cm}(2.2)$$
for all $x,y \in X.$ This means that $f$ is a quartic function.

b) Setting $x=y=0$ in (1.4) to obtain $f(0)=0.$ Putting $x=0$ in
(1.4), then by oddness of f, we have

$$f(2y)=2f(y), \eqno \hspace {0.5cm}(2.3)$$
for all $y\in X.$ We obtain from (1.4) and (2.3) that

$$f(2x+y)+f(2x-y)=4(f(x+y)+f(x-y))-4f(x), \eqno \hspace {0.5cm}(2.4)$$
for all $x,y\in X.$ Replacing y by -2y in (2.4), it follows that

$$f(2x-2y)+f(2x+2y)=4(f(x-2y)+f(x+2y))-4f(x). \eqno \hspace {0.5cm}(2.5)$$
Combining (2.3) and (2.5) to obtain

$$f(x-y)+f(x+y)=2(f(x-2y)+f(x+2y))-2f(x). \eqno \hspace {0.5cm}(2.6)$$
Interchange x and y in (2.6) to get the relation

$$f(x+y)+f(x-y)=2(f(y-2x)+f(y+2x))-2f(y). \eqno \hspace {0.5cm}(2.7)$$
Replacing y by -y in (2.7), and using the oddness of f to get

$$f(x-y)-f(x+y)=2(f(2x-y)-f(2x+y))+2f(y). \eqno \hspace {0.5cm}(2.8)$$
From (2.4) and (2.8), we obtain

$$4f(2x+y)=9f(x+y)+7f(x-y))-8f(x)+2f(y). \eqno \hspace {0.5cm}(2.9)$$
Replacing x+y by y in (2.9) it follows that

$$7f(2x-y)=4f(x+y)+2f(x-y))-9f(y)+8f(x). \eqno \hspace {0.5cm}(2.10)$$
By using (2.9) and (2.10), we lead to

\begin{align*}
f(2x+y)&+f(2x-y)=\frac{79}{28}f(x+y)+\frac{57}{28}f(x-y))\\
&-\frac{6}{7}f(x)-\frac{11}{14}f(y). \hspace {8cm}(2.11)
\end{align*}
We get from (2.4) and (2.11) that

$$3f(x+y)+5f(x-y))=8f(x)-28f(y). \eqno \hspace {0.5cm}(2.12)$$
Replacing x by 2x in (2.4) it follows that

$$f(4x+y)+f(4x-y)=16(f(x+y)+f(x-y))-24f(x). \eqno \hspace {0.5cm}(2.13)$$
Setting $2x+y$ instead of y in (2.4), we arrive at

$$f(4x+y)-f(y)=4(f(3x-y)+f(x-y))-4f(x).\eqno  \hspace {0.5cm}(2.14)$$
Replacing y by -y in (2.14), and using oddness of f to get

$$f(4x-y)+f(y)=4(f(3x+y)+f(x+y))-4f(x). \eqno \hspace {0.5cm}(2.15)$$
Adding (2.14) to (2.15) to get the relation

\begin{align*}
f(4x+y)+f(4x-y)&=4(f(3x+y)+f(3x-y))\\&-4(f(x+y)+f(x-y))-8f(x).
\hspace {4.2cm}(2.16)
\end{align*}
Replacing y by x+y in (2.4) to  obtain

$$f(3x+y)+f(x-y)=4(f(2x+y)-f(y))-4f(x). \eqno \hspace {0.5cm}(2.17)$$
Replacing y by -y in (2.17), and using the oddness of f, we lead to

$$f(3x-y)+f(x+y)=4(f(2x-y)+f(y))-4f(x). \eqno \hspace {0.5cm}(2.18)$$
Combining (2.17) and (2.18) to obtain

$$f(3x+y)+f(3x-y)=15(f(x+y)+f(x-y))-24f(x). \eqno \hspace {0.5cm}(2.19)$$
Using (2.16) and (2.19) to get

$$f(4x+y)+f(4x-y)=56(f(x+y)+f(x-y))-104f(x). \eqno \hspace {0.5cm}(2.20)$$
Combining (2.13) and (2.20), we arrive at

$$f(x+y)+f(x-y)=2f(x).\eqno \hspace {0.5cm}(2.21)$$
Hence by using (2.12) and (2.21) it is easy to see that f is
additive. This completed the proof of Theorem.
\end{proof}
\begin{thm}\label{t2}
Let $X$,$Y$ be vector spaces,  and let  $f:X\longrightarrow Y$  be a
function. Then f satisfies (1.4) if and only if there exist a unique
symmetric multi-additive function $B:X\times X\times X\times
X\longrightarrow Y$ and a unique additive function
$A:X\longrightarrow Y$ such that
 $f(x)=B(x,x,x,x)+A(x)$ for all $x\in X.$
\end{thm}
\begin{proof}
Let $f$ satisfies (1.4). We decompose f into the even part and odd
part by setting

$$f_e(x)=\frac{1}{2}(f(x)+f(-x)),~~\hspace {0.3 cm}f_o(x)=\frac{1}{2}(f(x)-f(-x)),$$
for all $x\in X.$ By (1.4), we have
\begin{align*}
f_e(2x+y)&+f_e(2x-y)=\frac{1}{2}[f(2x+y)+f(-2x-y)+f(2x-y)+f(-2x+y)]\\
&=\frac{1}{2}[f(2x+y)+f(2x-y)]+\frac{1}{2}[f(-2x+(-y))+f(-2x-(-y))]\\
&=\frac{1}{2}[4(f(x+y)+f(x-y))-\frac{3}{7}(f(2y)-2f(y))+2f(2x)-8f(x)]\\
&+\frac{1}{2}[4(f(-x-y)+f(-x-(-y)))-\frac{3}{7}(f(-2y)-2f(-y))+2f(-2x)-8f(-x)]\\
&=4[\frac{1}{2}(f(x+y)+f(-x-y))+\frac{1}{2}(f(-x+y)+f(x-y))]\\
&-\frac{3}{7}[\frac{1}{2}(f(2y)+f(-2y))-(f(y)-f(-y))]\\
&+2[\frac{1}{2}(f(2x)+f(-2x))]-8[\frac{1}{2}(f(x)+f(-x))]\\
&=4(f_e(x+y)+f_e(x-y))-\frac{3}{7}(f_e(2y)-2f_e(y))+2f_e(2x)-8f_e(x)
\end{align*}
for all $x,y\in X.$ This means that $f_e$ holds in (1.4). Similarly
we can show that $f_o$ satisfies (1.4). By above Theorem, $f_e$ and
$f_o$ are quartic and additive respectively. Thus there exists  a
unique symmetric multi-additive function $B:X\times X\times X\times
X\longrightarrow Y$ such that $f_e(x)=B(x,x,x,x)$ for all $x\in X.$
Put $A(x):=f_o(x)$ for all $x\in X.$ It follows that
$f(x)=B(x)+A(x)$ for all $x\in X.$ The proof of the converse is
trivially.
\end{proof}

\section{ Stability}

Throughout this section, X and Y will be a real normed space and a
real Banach space, respectively. Let $f:X\rightarrow Y$ be a
function then we define $D_f:X\times X \rightarrow Y$ by

\begin{align*}
D_{f}(x,y)&=7[f(2x+y)+f(2x-y)]-28[f(x+y)+f(x-y)]\\
&+3[f(2y)-2f(y)]-14[f(2x)-4f(x)]
\end{align*}
for all $x,y \in X.$

\begin{thm}\label{t2} Let $\psi:X\times X\rightarrow [0,\infty)$
be a function satisfies $\sum^{\infty}_{i=0}
\frac{\psi(0,2^ix)}{16^i}<\infty$ for all $x\in X$, and $\lim
\frac{\psi(2^n x,2^n y)}{16^n}=0$ for all $x,y\in X$. If
$f:X\rightarrow Y$ is an even function such that $f(0)=0,$ and that

$$\|D_f(x,y)\|\leq \psi(x,y), \eqno  \hspace {0.5cm}(3.1)$$ for all $x,y\in
X$, then there exists a unique quartic function $Q:X \rightarrow Y$
satisfying (1.4) and

$$\|f(x)-Q(x)\|\leq \frac{1}{48}\sum^{\infty}_{i=0} \frac{\psi(0,2^i x)}{16^i},\eqno \hspace {0.5cm}(3.2)$$
for all $x\in X$.
\end{thm}
\begin{proof}
 Putting $x=0$ in (3.1), then we have

$$\|3f(2y)-48f(y)\|\leq \psi(0,y). \eqno\hspace {0.5cm}(3.3)$$
Replacing y by x in (3.3) and then dividing by 48 to obtain

$$\parallel \frac{f(2x)}{16}-f(x)\parallel \leq
\frac{1}{48}\psi(0,x), \eqno\hspace {0.5cm}(3.4)$$ for all $x\in X.$
Replacing x by 2x in (3.4) to get

$$\parallel\frac{f(4x)}{16}-f(2x)\parallel\leq
\frac{1}{48}\psi(0,2x). \eqno\hspace {0.5cm}(3.5)$$ Combine (3.4)
and (3.5) by use of the triangle inequality to get

$$\parallel\frac{f(4x)}{16^2}-f(x)\parallel
\leq\frac{1}{48}(\frac{\psi(0,2x)}{16}+\psi(0,x)). \eqno\hspace
{0.5cm}(3.6)$$ By induction on $n\in \mathbb{N}$, we can show that

$$\parallel \frac{f(2^n x)}{16^ n}-f(x)\parallel\leq\frac {1}{48}
\sum^{n-1}_{i=0} \frac {\psi(0,2^i x)}{16^i}.\eqno \hspace
{0.5cm}(3.7)$$ Dividing (3.7) by $16^m$ and replacing x by $2^m x$
to get
\begin{align*}
\parallel \frac {f(2^{m+n}x)}{16^{m+n}}-\frac{f(2^{m}x)}{16^{m}}\parallel
&=\frac{1}{16^m}\parallel f(2^n 2^m x)-f(2^m x)\parallel\\
&\leq\frac {1}{48\times16^m}\sum^{n-1}_{i=0}\frac{\psi(0,2^i
x)}{16^i}\\
&\leq\frac{1}{48}\sum^{\infty}_{i=0}\frac{\psi(0,2^i 2^m
x)}{16^{m+i}},
\end{align*}
for all $x \in X$. This shows that $\{\frac{f(2^n x)}{16^n}\}$ is a
Cauchy sequence in Y, by taking the $\lim m\rightarrow \infty.$
Since Y is a Banach space, then the sequence $\{\frac{f(2^n
x)}{16^n}\}$ converges. We define $Q:X\rightarrow Y$ by
$Q(x):=\lim_n \frac{f(2^n x)}{16^n}$ for all $x\in X$. Since f is
even function, then Q is even. On the other hand we have
\begin{align*}
\|D_Q (x,y)\|&=\lim_n \frac{1}{16^n}\|D_f(2^n x, 2^n y)\|\\
&\leq\lim_n \frac{\psi(2^n x,2^n y)}{16^n}=0,
\end{align*}
for all $x,y\in X$. Hence by Theorem 2.1, Q is a quartic function.
To shows that Q is unique, suppose that there exists another quartic
function $\acute{Q}:X\rightarrow Y$ which satisfies (1.4) and (3.2).
We have $Q(2^n x)=16^n Q(x)$, and $\acute{Q}(2^n x)=16^n
\acute{Q}(x)$, for all $x\in X$. It follows that
\begin{align*}
\parallel \acute{Q}(x)-Q(x)\parallel &=\frac{1}{16^n}\parallel
\acute{Q}(2^n x)-Q(2^n x)\parallel\\
&\leq \frac{1}{16^n}[\parallel \acute{Q}(2^n x)-f(2^n
x)\parallel+\parallel f(2^n x)-Q(2^n x)\parallel]\\
&\leq\frac{1}{24}\sum^{\infty}_{i=0}\frac{\psi(0,2^{n+i}
x)}{16^{n+i}},
\end{align*}
for all $x\in X$. By taking $n\rightarrow \infty$ in this inequality
we have $\acute{Q}(x)=Q(x)$.
\end{proof}
\begin{thm}\label{t'2} Let $\psi:X\times X\rightarrow [0,\infty)$
be a function satisfies $\sum^{\infty}_{i=0}
16^i\psi(0,2^{-i-1}x)<\infty$ for all $x\in X$, and $\lim
16^n\psi(2^{-n} x,2^{-n} y)=0$ for all $x,y\in X$. Suppose that an
even function  $f:X\rightarrow Y$  satisfies f(0)=0, and (3.1). Then
the limit $Q(x):=\lim_n 16^n{f(2^{-n} x)}$ exists for all $x\in X$
and $Q:X \rightarrow Y$ is a unique quartic function
 satisfies (1.4) and

$$\|f(x)-Q(x)\|\leq \frac{1}{3}\sum^{\infty}_{i=0} 16^i\psi(0,2^{-i-1} x), \eqno \hspace {0.5cm}(3.8)$$
for all $x\in X$.
\end{thm}
\begin{proof} Putting $x=0$ in (3.1), then we have

$$\|3f(2y)-48f(y)\|\leq \psi(0,y). \eqno \hspace {0.5cm}(3.9)$$
Replacing y by  $\frac{x}{2}$  in (3.9) and result dividing by 3 to
get

$$\parallel 16f(2^{-1}x)-f(x)\parallel \leq
\frac{1}{3}\psi(0,2^{-1}x),\eqno \hspace {0.5cm}(3.10)$$ for all
$x\in X.$ Replacing x by $\frac{x}{2}$ in (3.10) it follows that

$$\parallel 16f(4^{-1}x)-f(2^{-1}x)\parallel\leq
\frac{1}{3}\psi(0,2^{-2}x).\eqno \hspace {0.5cm}(3.11)$$ Combining
(3.10) and (3.11) by use of the triangle inequality to obtain

$$\parallel 16^2f(4^{-1}x)-f(x)\parallel
\leq\frac{1}{3}(\frac{\psi(0,2^{-2}x)}{16}+\psi(0,2^{-1}x)). \eqno
\hspace {0.5cm}(3.12)$$ By induction on $n\in \mathbb{N}$, we have

$$\parallel 16^n f(2^{-n} x)-f(x)\parallel\leq\frac {1}{3}
\sum^{n-1}_{i=0} 16^i\psi(0,2^{-i-1} x).\eqno \hspace
{0.5cm}(3.13)$$ Multiplying (3.13) by $16^m$ and replacing x by
$2^{-m} x$ to obtain
\begin{align*}
\parallel 16^{m+n} {f(2^{-m-n}x)}-16^m{f(2^{-m}x)}\parallel
&={16^m}\parallel f(2^{-n} 2^{-m} x)-f(2^{-m} x)\parallel\\
&\leq\frac {16^m}{3}\sum^{n-1}_{i=0}16^i{\psi(0,2^{-i-1}
x)}\\
&\leq\frac{1}{3}\sum^{\infty}_{i=0}{16^{m+i}}{\psi(0,2^{-i-1} 2^{-m}
x)},
\end{align*}
for all $x \in X$. By taking the $\lim m\rightarrow \infty,$ it
follows that $\{16^n{f(2^{-n} x)}\}$ is a Cauchy sequence in Y.
Since Y is a Banach space, then the sequence $\{16^n{f(2^{-n} x)}\}$
converges. Now we define $Q:X\rightarrow Y$ by $Q(x):=\lim_n
16^n{f(2^{-n} x)}$ for all $x\in X$. The rest of proof is similar to
the proof of Theorem 3.1.
\end{proof}

\begin{thm}\label{t2} Let $\psi:X\times X\rightarrow [0,\infty)$ be
a function such that $$\sum \frac{\psi(0,2^i x)}{2^i}< \infty, \eqno
\hspace {0.5cm}(3.14)$$ and
$$\lim_n \frac {\psi(2^n x,2^n y)}{2^n}=0, \eqno \hspace {0.5cm}(3.15)$$ for
all $x,y \in X$. If $f:X\rightarrow Y$ is an odd function such that
$$\|D_f (x,y)\|\leq \psi(x,y), \eqno \hspace {0.5cm}(3.16)$$
for all $x,y\in X$. Then there exists a unique additive function
$A:X\rightarrow Y$ satisfies  (1.4) and
$$\|f(x)-A(x)\|\leq\frac{1}{2} \sum^{\infty}_{i=0}\frac{\psi(0,2^i
x)}{2^i},$$ for all $x\in X$.
\end{thm}

\begin{proof} Setting $x=0$ in (3.16) to get

$$\|f(2y)-2f(y)\|\leq \psi(o,y). \eqno \hspace {0.5cm}(3.17)$$
Replacing y by x in (3.17) and result dividing by 2, then we have

$$\|\frac{f(2x)}{2}-f(x)\|\leq\frac{1}{2} \psi(0,x). \eqno\hspace {0.5cm}(3.18)$$
Replacing x by 2x in (3.18) to obtain

$$\|\frac{f(4x)}{2}-f(2x)\|\leq\frac{1}{2} \psi(0,2x). \eqno\hspace {0.5cm}(3.19)$$
Combine (3.18) and (3.19) by use of the triangle inequality to get

$$\|\frac{f(4x)}{4}-f(x)\|\leq\frac{1}{2} (\psi(0,x)+\frac{1}{2}\psi(0,2x)).\eqno\hspace {0.5cm}(3.20)$$
Now we use iterative methods and induction on $n$ to prove our next
relation.
$$\|\frac{f(2^n x)}{2^n}-f(x)\|\leq\frac{1}{2} \sum^{n-1}_{i=0}\frac{\psi(0,2^i x)}{2^i}.\eqno\hspace {0.5cm}(3.21)$$
Dividing (3.21) by $2^m$ and then substituting x by $2^m x$, we get

\begin{align*}
\parallel \frac {f(2^{m+n}x)}{2^{m+n}}-\frac{f(2^{m}x)}{2^{m}}\parallel
&=\frac{1}{2^m}\parallel \frac{f(2^n 2^m x)}{2^n}-f(2^m x)\parallel\\
&\leq\frac {1}{2^{m+1}}\sum^{n-1}_{i=0}\frac{\psi(0,2^i
2^m x)}{2^i}\\
&\leq\frac{1}{2}\sum^{\infty}_{i=0}\frac{\psi(0,2^{i+m}x)}{2^{m+i}}\hspace
{5.5cm}(3.22)
\end{align*}
Taking $m\rightarrow \infty$ in (3.22), then the right hand side of
the inequality tends to zero. Since Y is a Banach space, then
$A(x)=\lim_n \frac {f(2^n x)}{2^n}$ exits for all $x\in X$. The
oddness of f implies that A is odd. On the other hand by (3.15) we
have
\begin{align*}
D_A (x,y)=\lim_n \frac{1}{2^n}\|D_f(2^n x,2^n y)\|\leq\lim_n
\frac{\psi(2^n x,2^n y)}{2^n}=0.
\end{align*}
Hence  by Theorem 1.2, A is additive function. The rest of the proof
is similar to the proof of  Theorem 3.1.
\end{proof}

\begin{thm}\label{t''2} Let $\psi:X\times X\rightarrow [0,\infty)$
be a function satisfies $$\sum^{\infty}_{i=0}
2^i\psi(0,2^{-i-1}x)<\infty,$$ for all $x\in X$, and $\lim
2^n\psi(2^{-n} x,2^{-n} y)=0$ for all $x,y\in X$. Suppose that an
odd function  $f:X\rightarrow Y$  satisfies (3.1). Then the limit
$A(x):=\lim_n 2^n{f(2^{-n} x)}$ exists for all $x\in X$ and $A:X
\rightarrow Y$ is a unique additive function satisfying (1.4) and

$$\|f(x)-A(x)\|\leq \sum^{\infty}_{i=0} 2^i\psi(0,2^{-i-1} x)$$
for all $x\in X$.
\end{thm}
\begin{proof}
It is similar to the proof of Theorem 3.3.
\end{proof}

\begin{thm}\label{t5} Let $\psi:X\times X\rightarrow Y$ be a
function such that
$$\sum^{\infty}_{i=o} \frac{\psi(0,2^i x)}{2^i}\leq
\infty \quad and \quad \lim_n \frac{\psi(2^n x,2^n x)}{2^n}=0,$$ for
all $x\in X$. Suppose that a function $ f:X\rightarrow Y$ satisfies
the inequality $$\|D_f (x,y)\|\leq \psi(x,y),$$ for all $x,y\in X$,
and $f(0)=0$. Then there exist a unique quartic function
$Q:X\rightarrow Y$ and a unique additive function $A:X\rightarrow Y$
satisfying (1.4) and
\begin{align*}
\parallel f(x)-Q(x)-A(x)\parallel
&\leq\frac {1}{48}[\sum^{\infty}_{i=0}(\frac{\psi(0,2^i x)+\psi(0,-2^i x)}{2\times16^i}\\
&+\frac{12(\psi(0,2^i x)+\psi(0,-2^i x))}{2^i})], \hspace
{4.2cm}(3.23)
\end{align*}
for all $x,y\in X$.
\end{thm}

\begin{proof} We have
$$\|D_{f_e} (x,y)\|\leq\frac{1}{2}[\psi(x,y)+\psi(-x,-y)]$$
for all $x,y\in X$. Since $f_e(0)=0$ and $f_e$ is and even function,
then by Theorem 3.1, there exists a unique quartic function
$Q:x\rightarrow Y$ satisfying
$$\parallel f_e(x)-Q(x)\parallel \leq\frac {1}{48}\sum^{\infty}_{i=0} \frac{\psi(0,2^i x)+\psi(0,-2^i
x)}{2\times16^i}, \eqno\hspace {0.5cm}(3.24)$$ for all $x\in X$. On
the other hand $f_0$ is odd function and $$\|D_{f_0} (x,y)\|\leq
\frac{1}{2}[\psi(x,y)+\psi(-x,-y)],$$ for all $x,y\in X$. Then by
Theorem 3.3, there exists a unique additive function $A:X\rightarrow
Y$ such that
$$\parallel f_0(x)-A(x)\parallel \leq\frac {1}{2}\sum^{\infty}_{i=0} \frac{\psi(0,2^i x)+\psi(0,-2^i
x)}{2\times2^i}, \eqno\hspace {0.5cm}(3.25)$$ for all $x\in X$.
Combining (3.24) and (3.25) to obtain (3.23). This completes the
proof of Theorem.
\end{proof}

By Theorem 3.5, we are going to investigate the Hyers-Ulam -Rassias
stability problem for functional equation  (1.4).

\begin{cor}\label{t2}
 Let $\theta\geq0$, $P<1$. Suppose $f:X\rightarrow Y$ satisfies the
inequality
$$\|D_f (x,y)\|\leq\theta(\|x\|^p+\|y\|^p),$$
for all $x,y\in X$, and $f(0)=0$. Then there exists a unique quartic
function $Q:X\rightarrow Y$, and a unique additive function
$A:X\rightarrow Y$ satisfying (1.4), and
$$\parallel f(x)-Q(x)-A(x)\parallel \leq\frac {\theta}{48}\|x\|^p
(\frac{16}{16-2^p}+\frac{96}{1-2^{p-1}}),$$ for all $x\in X$.
\end{cor}

By Corollary 3.6, we solve the following Hyers-Ulam stability
problem for functional equation  (1.4).

\begin{cor}\label{t2} Let $\epsilon$ be a positive real number, and let $f:X\rightarrow Y$ be a function satisfies $$\|D_f
(x,y)\|\leq\epsilon,$$ for all $x,y\in X$. Then there exist a unique
quartic function $Q:X\rightarrow Y$, and a unique additive function
$A:X\rightarrow Y$ satisfying (1.4) and
$$\parallel f(x)-Q(x)-A(x)\parallel \leq\frac {362}{45} ~\epsilon,$$
for all $x\in X$.
\end{cor}

By applying Theorems 3.2 and 3.4, we have the following Theorem.

\begin{thm}\label{t5} Let $\psi:X\times X\rightarrow Y$ be a
function such that
$$\sum^{\infty}_{i=o} 16^i{\psi(0,2^{-i-1} x)}\leq
\infty \quad and \quad \lim_n 16^n{\psi(2^n x,2^n x)}=0,$$ for all
$x\in X$. Suppose that a function $ f:X\rightarrow Y$ satisfies the
inequality $$\|D_f (x,y)\|\leq \psi(x,y),$$ for all $x,y\in X$, and
$f(0)=0$. Then there exist a unique quartic function $Q:X\rightarrow
Y$ and a unique additive function $A:X\rightarrow Y$ satisfying
(1.4) and
\begin{align*}
\parallel f(x)-Q(x)-A(x)\parallel
&\leq\sum^{\infty}_{i=0}[(\frac{16^i}{3}+2^i)(\frac{\psi(0,2^{-i-1}
x)+\psi(0,-2^{-i-1} x)}{2})],
\end{align*}
for all $x,y\in X$.
\end{thm}

\begin{cor}\label{t2}
 Let $\theta\geq0$, $P>4$. Suppose $f:X\rightarrow Y$ satisfies the
inequality
$$\|D_f (x,y)\|\leq\theta(\|x\|^p+\|y\|^p),$$
for all $x,y\in X$, and $f(0)=0$. Then there exist a unique quartic
function $Q:X\rightarrow Y$, and a unique additive function
$A:X\rightarrow Y$ satisfying (1.4), and
$$\parallel f(x)-Q(x)-A(x)\parallel \leq\frac {\theta}{3\times 2^p}\|x\|^p
(\frac{1}{1-2^{4-p}}+\frac{1}{1-2^{1-p}}),$$ for all $x\in X$.
\end{cor}

{\small


}
\end{document}